\documentclass[twoside,12pt]{article}

\usepackage{amsfonts}
\usepackage{amssymb}
\usepackage{amsmath}
\usepackage{latexsym}
\usepackage{amsthm}
\usepackage{amsfonts}
\usepackage{graphicx}
\usepackage{amsthm}

\newtheorem{theorem}{Theorem}[section]
\newtheorem{corollary}[theorem]{Corollary}
\newtheorem{conjecture}[theorem]{Conjecture}
\newtheorem{lemma}[theorem]{Lemma}

\newtheorem{definition}[theorem]{Definition}

\title{Some heterochromatic theorems for matroids}
 
\author{Criel Merino \and Juan Jos\'e Montellano-Ballesteros} 
\date{\today}
 
\begin{document}
 \maketitle
 \thispagestyle{empty} 

\begin{abstract}
	   The anti-Ramsey number of Erd\"os, Simonovits and S\'os from 1973 has become a classic invariant in Graph Theory. 
	   To study this invariant in Matroid Theory, we use a related invariant introduce by Arocha, Bracho and Neu\-mann-\-Lara.
	    The  heterochromatic number $hc(H)$ of a non-empty hypergraph $H$ is the smallest integer $k$ such that for every
	     colouring of the vertices of $H$ with exactly $k$ colours, there is a totally multicoloured hyperedge of $H$.
	     
 	Given a rank-$r$ matroid $M$, there are several hypergraphs associated to the matroid that we can consider. One is $C(M)	$, the hypergraph where the points are the elements of the matroid and the hyperedges are the circuits of $M$. 
	The other one is $B(M)$, where here the points are the elements and the hyperedges are the bases of the matroid.
 
        We prove that $hc(C(M))$  equals  $r+1$ when $M$ is not the free matroid $U_{n,n}$, and that if $M$ is a paving 
        matroid, then $hc(B(M))$ equals $r$. Then we explore the case when the hypergraph has the Hamiltonian circuits
         of the matroid as hyperedges, if any, for a class of paving matroids. We also extend the trivial observation of Erd\"os, 
         Simonovits and S\'os for the anti-Ramsey number for 3-cycles to 3-circuits in  projective geometries over finite fields.
\end{abstract}

\section{Introduction}

For an interger $n$ and a graph $H$, the classical anti-Ramsey number $ar(n,H)$ is the maximum number of 
 colours in an edge colouring of a complete graph $K_n$ with no totally multicoloured copy of $H$. It was introduced by
  Erd\"os, Simonovits and S\'os in~\cite{ErSiSo1973}. 
  
  In that paper they proved that $ar(n,C_3)=n-1$ and conjectured that
  $$
      ar(n,C_p)=n\left( \frac{p-2}{2}+\frac{1}{p-1}\right) +O(1)
      $$
      This was proved by Montellano and Neumann-Lara in 2005, see~\cite{MoNe2005}. It is worth mentined that 
      in~\cite{ErSiSo1973}, the proof of the $ar(n,C_3)$ is consider trivial, however is quite beautiful and it has inspired us to write 
      an extension to projective geometries. 
      
      The study of anti-Ramsey number for types of graphs $H$ include cycles~\cite{Alon1983,MoNe2000,JiWe2003}, 
      paths~\cite{SiSo1984},  $t$ disjoint $K_2$ or  matchings~\cite{Schiermeyer2004,FuKaScSu2009,ChLiTu2009}, 
      $t$ disjoint cycles~\cite{JiLi2009}, 
      trees~\cite{JiWe2004}, cliques~\cite{ErSiSo1973,MoNe2002,Schiermeyer2004}, 
      stars~\cite{Jiang2002,Montellano2006,MaSpTuVo1996} and
       others. For a survey of the results of these clases and numerous generalizations of the anti-Ramsey number we 
       refereed the reader to~\cite{FuMaOz2014}.
          
      One possible generalization was introduced by  Arocha, Bracho and Neu\-mann-\-Lara  in~\cite{ArBrNe1992} 
      when they defined the heterochromatic number for hypergraphs. 
      A $t$-coloring of a hypergraph $H=(V,E)$  is an onto function from $V$ to the set $[t]$. That is, a $t$-colouring 
      has to use all the colours. A $t$-coloring $\phi$ of $H$ is {\it heterocromatic} if there exist some hyperedge $e\in E$ 
      such that the image of the vertices in  $e$ under $\phi$ are all distinct. That is, a $t$-colouring is heterochromatic if 
      there is a totally multicoloured hyperedge. 
 
   \begin{definition}
   	For a hypergraph $H$ we define  $hc(H)$ to be the minimum value  $t$ for which all $t$-coloring of  $H$ is heterocromatic. 
   \end{definition}
   
   Clearly, $hc(H)\leq |V|$ and if $H'$ is an spanning sub-hypergraph of $H$, then $hc(H')\geq hc(H)$. 
   In~\cite{ArBrNe1992}, the auhors were interested in the extremal case when $hc(H)$ is as small as posible and 
   made the following definition. When the hypergraph has as vertices the edges of the complete graph  $K_n$ and the 
   hyperedges are the set of edges of the a subgraph isomorphic to a fixed graph $H$, the heterocromatic number of the hypergraph is the anti-Ramsey number $ar(n,H)$ plus one. From now on, we concentrate on the heterocromatic number. 
 
   \begin{definition}
	 An $r$-uniforme  $H$ hypergraph is tight if $hc(H)=r$. 
   \end{definition}
 
   Here he study the problem of finding the heterochromatic number for hypergraphs coming from matroids. 
   This has been consider before in~\cite{MoRi2013} but just partially. Here we concentrate on  hypergraph of basis  for  
   paving matroids in Section~\ref{paving} and the hypergraph of circuits in general matroids in Section~\ref{sec:circuits}. 
   This later result gives a simpler proof of the corresponding result for graphs stated in~\cite{Montellano2006}.
   We also consider the corresponding problem of $p$-circuits in matroids, but this is a much difficult problem, 
   that is not surprising as the corresponding problem for the complete graph took 30 years to be solved. However, 
   we did make some advances in the case of projective planes.  In Section~\ref{steiner}, we manage to extend this results 
   to a no trivial class 
   of paving matroids that are obtained from Steiner systems. 

\section{The hypergraph of circuits of a matroid}\label{sec:circuits}

 We start our study of hypergraphs associated to matroids by finding the heterocromatic number of the hypergraphs of circuits of a matroids. We denote by $C(M)$ the hypergraph of circuits of a matroid $M$, that is, the vertices in the hypergraph are the elements of the matroid, and the hyperedges are the circuits.
 
 \begin{theorem}\label{Th:hete_circuit}
 For a loopless rank-$r$ matroid $M$ with $n$ elements that is not the free matroid $U_{n,n}$, we  have that $hc(C(M))=r+1$
 \end{theorem}
 \begin{proof}
 Take any colouring of the elements of $M$ with $r+1$ colours. We use all the colours as $M$ is not $U_{n,n}$ and $n\geq r+1$. Take the an heterochromatic subset $A$ of size $r+1$. Then,  $A$ is dependent and contains a circuit.  So there is an heterochromatic  circuit $C$. 
 
 Now, if we have less than $r+1$ colours, take a cocircuit $C^*$ and colour all its elements with color $1$. Because for any circuit $C$, $|C\cap C^*|\neq 1$, any circuit $C$ with non-empty intersection with $C^*$ will have at least 2 elements of the same colour.  But $M'=M\setminus C^*$ is a loopless rank $r-1$ matroid (it  is actually a hyperplane). By induction we can colour $M'$ with less than $r$ colours such that no circuit is heterochromatic. At the base of the induction we have loopless rank-1 matroids, where the result is trivial. Then, the results follows. 
 \end{proof}
 
 Note that if $M$ has a loop, then $hc(C(M))$ equals 1, while if $M$ is isomorphic to $U_{n,n}$, the hypergraph $C(M)$ is does not have hyperedges.  For the hypergraph $C^*(M)$, whose points are the elements of the matroid $M$ and the hyperedges are the cocircuits of $M$, by using duality, we have the following 

\begin{corollary} 
For a coloopless rank-$r$ matroid $M$ with $n$ elements that is not the uniform matroid $U_{0,n}$, we  have that $hc(C^*(M))=r^*+1$
\end{corollary}

 These results generalised the case for graphs. For a graphic matroid $M(G)$ of a connected graph $G$, let $C(G)=C(M(G))$ the hypergraph whose vertices are the edges of a graph $G$ and whose hyperedges are the cycles of $G$ and let $C^*(M(G))=C^*(G)$ the hypergraph of the edge-cuts of $G$.  Thus, for a connected graph $G$ with $n$ vertices and $m$ edges, $hc(C(G))=n $ and $hc(C^*(G))= m-n+2$. This last results was originally proved in~\cite{Montellano2006}.

 \section{The hypergraph of bases of a matroid}\label{paving}
 
     For the second theorem we let $B(M)$ be the hypergraph whose vertices are the elements of $M$ and its hyperedges the 
     bases of $M$. Given a family $\mathcal{B}$ of subsets of a set $E$, a subset $C$ is a  2-{\em transversal} 
     if the intersection of $C$ with any member of $\mathcal{B}$ has cardinality at least 2. For a rank $r$ matroid, 
     we called a rank $r-2$ flat a 
     \emph{coline}. Clearly, the complement of a coline 
     is a 2-transversal of the clutter of bases of $M$, thus, if $t_2$ is the size of the smallest complement of a coline and   
     $M$ has $n$ elements, $hc(B(M))$ is at least $n-t_2+2$. This lower bound is actually the right value. 
 
 \begin{theorem}\label{Th:hc_bases}
 For a rank-$r$ matroid $M$ with $n$ elements, we  have that $hc(B(M))=n-t_2+2$
 \end{theorem}
 \begin{proof}
  Let $t$ be the size of $L$, the coline in $M$ with maximum cardinality,  then $t$ is the size of the smallest complement of a
   coline. As $E(M)\setminus L$ is a 2-transversal, it is enough to prove that $hc(B(M))\leq t+2$. To this end, we give a 
   $t+2$ colouring of $B(M)$ and prove that there is an heterochromatic basis. 
  
  Let $A$ be an heterochromatic  set of $t+2$ points. As $L$ is a coline, $t+2$ is greater or equal $r$, the rank of the matroid. If the rank of $A$ is $r$, we are done. If the rank is $r-1$, let take $H$ a hyperplane that contains $A$ and let $a$ be a  point that is not in $H$. Let assume that the colour of $a$ is 1 and consider $B$ the set of points in $A$ of colour different from 1, that is $B=A\setminus\{a\}$. The set $B$ has cardinality $t+1$ that is bigger than the size of any rank $r-2$ flat, thus $B$ has rank $r-1$. Thus, $B\cup \{a\}$ has rank $r$ and its heterochromatic. Finally, $A$ cannot have rank less than $r-1$ as its bigger than any rank $r-2$ flat. 
 \end{proof}
 
    By Corollary~9 of~\cite{MoRi2013}, the largest complement of a coline is a minimal 2-transversal. 
    For some families of matroids 
    this is easy to compute. For the graphic matroid of the complete graph $K_n$, the rank is $n-1$ and the largest rank 
    $n-3$ flats are complete subgraphs $K_{n-2}$. We can, then, compute the heterochromatic number as 
    ${n-2 \choose 2}+2={n \choose 2}-2n+5={n \choose 2}-(2n-3)+2 $ that is a results given in \cite{JiWe2004}. 
    For the matroids coming from the projective spaces PG($r-1,q$) the heterochromatic number of bases is 
    $|PG($r-3,q$)|$+2. In particular, for projective planes, this value is 3. And for the $n$ wheels, $W_n$ the value 
    is $(2n-5)-2=2n-3$ 
 
 \subsection{Paving matroids}

 Paving matroid are matroids are matroids such that the circuits are at least as big as the rank of the matroid. And sparse paving matroids are paving matroids such that any two circuits of size $r$  $C$ and $D$ are far apart, that is $|C\bigtriangleup D|>2$. Paving and sparse paving are  a potentially large clase of matroids. Example of paving matroids are uniform matroids, projective planes and matroids coming from Steiner systems $S(r-1,d,n)$.  
 
 \begin{theorem}
 if $M$ is a rank-$r$ paving matroid, then $hc(B(M))=r$.
 \end{theorem}
 \begin{proof}
   If $M$ is a uniform matroid the result is trivial (any set of $r$ elements is a base) so we assume that $M$ is a non-uniform paving matroid. Take any colouring of the elements of $M$ with $r$ colours. Take a totally multicoloured set of $r$ elements, one per colour class, and named this set $A$. As $A$ has $r$ elements, is either a basis or a circuit (as $M$ is paving). In the former case we are done. If $A$ is a circuit, take an element $a$ not in the closure of $A$, this is possible as $A$ has rank $r-1$. Suppose that $a$ has colour 1, then we can use the elements in $A$ with colour different to 1 to form a basis. 
   
   Clearly, a colouring with $r-1$ colours of the elements of $M$ will forbid any basis to be heterochromatic.
 \end{proof}
 
  As a corollary, we get that the heterochromatic number for the hypergraph of the bases of a projective planes is $3$. Also, the result implies that the hypergraphs $B(M)$ are tight. A whole classification of tight hypergraphs is a tantalizing problem, but in the case of hypergraphs $B(M)$, the next result completely solves the problem. 
 
  \begin{theorem}
 if $M$ is a rank-$r$  matroid such that $hc(B(M))=r$, then $M$ is paving.
 \end{theorem}
 \begin{proof}
    Let $M$ be a non-paving matroid. We are going to show that there is a $r$-colouring of $B(M)$ that is not 
    heterochromatic, that is, that no basis is totally multicoloured.  As $M$ is not paving, it has a circuit of sizes $k<r$, say 
    $C$. We totally multicolour the elements of $C$ with, necessarily, $k$ different colours. We colour the remaining elements
     in $M$ with the unused $r-k$ colours. 
    Then, any $r$-subset disjoint from $C$ is not totally multicoloured. Also, any independent subset  cannot use more than 
    $k-1$ elements from $C$. Thus, any totally multicoloured independent  subset will have at most $r-k+(k-1)=r-1$ 
    elements. So $hc(B(M)>r$.
  \end{proof}

    The last two theorems say that $B(M)$ is tense iff $M$ is paving. An alternative approach to prove the last two results is 
    to use  the general Theorem~\ref{Th:hc_bases}, as all the rank-($r-2$) flats in a paving matroid have cardinality $r-2$. 
    Also if a  matroid is not paving, there is at least one rank-($r-2$) flat of size bigger than $r-2$. 

\section{Small and large circuits}\label{steiner}
 
   Projective planes are an interesting class of paving matroids that play a similar role  to complete graphs for rank-2
    representable matroids.  By the results on the previous section we get that the heterochromatic number of the 
    hypergraph of the bases of a projective plane $PG(2,q)$ of order $q$ is 3. Now, being a paving matroid, $PG(2,q)$ 
    has circuits of size 3 and 4, the former being 3 points in a line and the latter being 4 points in general position. Circuits of the 
    later form are called hamiltonian. For a rank-$r$ matroid $M$, we call a circuit of size $r+1$ a \emph{Hamiltonian circuit} and 
    $M$ is  called a Hamiltonian matroid. These matroids  have been considered in~\cite{Borowiecki1983}. 
    
    Let us start with small circuits and define $C_3(PG(2,q))$ to be the hypergraphs of 3-circuits. 
  
  \begin{theorem}\label{th:proj_3_circuit_rank_2}
  For a projective plane of order $q$, $PG(2,q)$, we have that $hc(C_3(PG(2,q)))=4$. 
 \end{theorem}
 \begin{proof}
    Because the rank of the matroid associated to $PG(2,q)$ has rank 3, we know by Theorem~\ref{Th:hete_circuit} that 
    with 3 colours there is a 3-colouring that has no heterocromatic circuit, so $hc(C_3(PG(2,q)))\geq4$. 
    Let us take a 4-colouring of the points in $PG(2,q)$. 
    Again, by Theorem~\ref{Th:hete_circuit},  we know any 4-colouring  is  heterocromatic, that is, there is a totally coloured
     circuit. If this circuit  is  3 points in a line we are done. If the circuit are 4 points in general position, let us label the points
      $a$, $b$, $c$ and $d$. Now, take the lines $ab$ and $cd$, as we are in a projective plane, there is a point $e$ in the 
      intersection of the two lines. This point is neither of the previous 4 points. Whichever the colour of $e$ is, exactly one 
      of the pairs $\{a,b\}$ or $\{c,d\}$ don't contain that colour. Thus, we get 3 points in a line all with different colour. 
  \end{proof}
  
   In a similar fashion, the projective space $PG(r-1,q)$ play the same role as complete graphs for representable matroids over
   the field $GF(q)$. Let  $C_3(PG(r-1,q))$ be the hypergraphs of 3-circuits of the matroid $PG(r-1,q)$. 
   
    \begin{theorem}\label{th:proj_3_circuit_rank_r}
  For the rank-$r$ projective geometry over $GF(q)$, $PG(r-1,q)$, we have that $hc(C_3(PG(r-1,q)))=r+1$. 
 \end{theorem}
 \begin{proof}
   By Theorem~\ref{Th:hete_circuit}, we have the lower bound $hc(C_3(PG(r-1,q)))\geq r+1$. We will show that any $r+1$-colouring of the hypergraph $C_3(PG(r-1,q))$ has an heterochromatic edge using the same argument as in~\cite{ErSiSo1973}.
   
   The case when the rank equals 2 is Theorem~\ref{th:proj_3_circuit_rank_2}.  Let us suppose that $r>2$. Take a totally multicoloured set $A$ with $r+1$ colours in $PG(r-1,q)$. The set $A$ is necessarily dependent, so there is a circuit that is totally multicolour. Let $C$ be a totally multicoloured circuit of minimum size. If this circuit has size 3, we are done. Suppose that $r+1>|C|=k>3$, and consider the flat $F$ span by $C$. As we are in a projective geometry, $F$ is a projective geometry of rank $k-1$ over $GF(q)$ which elements has been $k$ colour. Then the result follows by induction. 
   
    We can suppose that $|C|=r+1$ and then $C$ is a Hamiltonian circuit. Let the elements of $C$ be $e_1$,$\ldots$, $e_{r+1}$.
    Take the line $L$ be the flat span by $e_1$ and $e_2$, and $H$ be the hyperplane span by the remaining $r-1$ elements. In a projective space, if a line $L'$ is not contained in a hyperplane $H'$, $L'\cap H'$ is exactly one point. Let $a$ be the point at the intersection of $L$ and $H$. This point cannot be a point already in $C$ as circuits are minimal dependent sets. Thus $\{a,e_1,e_2\}$ is a 3-circuit $C_2$ and 
    $\{a,e_3,\ldots,e_{r+1}\}$ contains a circuit $C_3$. One of the circuits is totally multicoloured and we have a contradiction with the election of $C$.    
 \end{proof}
 
  Probably the reader would disagree that the previous proof is the same as in~\cite{ErSiSo1973}, but this impression
   is due to the strong intuition given by the vertex structure in a graph. In order to prove our point, we invite the reader to rewrite 
    the final part of the above argument in the case of the field $GF(2)$ using the following lemma. 
   
   \begin{lemma}
    Let $M$ be a rank-$r$ simple Hamiltonian binary matroid that is not $U_{r-1,r}$. Then, any $r+1$ colouring of the elements of the matroid such that a Hamiltonian circuit is totally multicoloured contains a strictly smaller totally multicolored circuit. 
    \end{lemma}
     \begin{proof}
   Let $C$ be the totally multicoloured Hamiltonian circuit and let $a$ be an element of $M$ not in $C$. 
   The subset $C+a$ contains  a circuit $a\in C'\neq C$ and $|C'|<|C|$ because $M$ is simple. If $C'$ is totally 
    multicoloured, we are done. Else, $C'$ has an element $a'\neq a$ with
    the same colour as $a$.  Take now $C'\triangle C$, that is a disjoint union of circuits because $M$ is binary. Let $D$ 
    be the circuit in this union that contains $a$ but this time, $a'$ is not in $D$ and $D$ is totally multicoloured and strictly 
    smaller than $C$. 
 \end{proof}

    Hamiltonian cycles play an important role in the main theorem in~\cite{MoNe2005}.  Let $HC(M)$ 
    be the hypergraph of Hamiltonian circuits of the matroid  $M$ when the matroid is Hamiltonian.
   We have the following general upper bound for the heterocromatic number of  $HC(M)$.

  \begin{lemma}\label{lemma:hamil_circuits}
    Let $M$ be a rank-$r$  matroid with at least one Hamiltonian circuit and let $q$ be the size of the largest hyperplane 
    in $M$, then we have that $hc(HC(M))> q+1$
   \end{lemma}
   \begin{proof}
     Let $H$ be a hyperplane of largest size and $C$ any circuit of size $r+1$.  Then $|C\cap H|\leq r-1$, and so, 
     the complement of $H$ in $E(M)$ is a 2-transversal. Thus, the result follows.  
  \end{proof}
  
   Projective planes are Hamiltonian matroids, thus, by Lemma~\ref{lemma:hamil_circuits}, the heterochromatic number of a projective plane of order $q$ is greater than $q+2$. The following theorem proves that the correct value is in fact $q+3$

\begin{theorem}
  For a projective plane of order $q>1$, $PG(2,q)$, we have that $hc(HC(PG(2,q)))=q+3$. 
 \end{theorem}
 \begin{proof}
    Because $q+3>4$, by Theorem~\ref{th:proj_3_circuit_rank_2}, there is a line $L$ with 3 points, say $a$, $b$ and $c$, that have all
     different colour, say 1,2 and 3, respectively. As the points in $L$ use at most $q+1$ colours, there are two points not in $L$,
      say $d$ and $e$, that use two unused colours, say 4 and 5.  Take the line $L'$ that passes through $d$ and $e$. 
      Because $L$ and $L'$ intersect in exactly one point, one of the pairs, $\{a,b\}$, $\{a,c\}$ or $\{b,c\}$ together 
      with $\{d,e\}$ form a totally multicoloured Hamiltonian circuit.
  \end{proof}

We can say more for the heterocromatic number of Hamiltonian circuits in a particular class of paving matroid. Let us call a rank-$r$  matroid, a \emph{perfect matroid design}, if all its flats of rank $k$ have the same cardinality, for $1\leq k\leq r$. Observe that paving matroids satisfy the condition of being perfect matroid designs  trivially for all values of $k$ different from $r-1$. Thus, when we consider paving matroids that are perfect matroid designs, we are really asking just that the size of all hyperplanes is the same, this is precise the case when the paving matroid is a Steiner system $S(r-1,d,n)$. In fact, the family of matroids in the intersection of paving and perfect matroid designs are precisely the ones coming from  Steiner systems. The reason for this is that the type of a paving matroid that is a perfect matroid design is the one of a Steiner system, and it is known that matroids coming from Steiner systems are characterised by its type. 

The following results was mentioned by David Pike in a private communication, we include his prove for completeness. 
\begin{theorem}
 For a Steiner system $S(t,k,n)$,  with $2\leq t< k< n$, whenever $n > (t+1)(k+1-t)$ then the associated rank-$(t+1)$ paving matroid  has a Hamiltonian circuit.
\end{theorem}
\begin{proof}
Start with any fixed $t$-set $T=\{x_1,x_2,\ldots,x_t\}$, which defines one block $B$.
Let $y$ be any point not in $B$, and so $T \cup \{y\}$ is a base set of size $t+1$.  We wish to now prove that there exists a point $z$ such that $T \cup \{y,z\}$ has no circuits of size $t+1$.

Consider the $t$ distinct blocks that are induced by the $t$ $t$-sets of the form $\{y\} \cup (T \setminus \{x_i\})$ where $i=1,2,\ldots,t$.  Together with the block $B$, we have $t+1$ blocks of size $k$ that contain at most $(t+1)(k+1-t)$ distinct points.  So if there is at least one other point, then such a point can be join with $T \cup {y}$ to produce a ($t+2$)-set that contains no circuits.
 \end{proof}

 \begin{theorem}\label{lemma:r+1_circuits}
 If $M$ is a rank-$r$ paving matroid coming from a Steiner system $S(r-1,q,n)$ with $r$ odd, we have that $hc(HC(M))= q+2$
 \end{theorem}
 \begin{proof}
    Lets take $h$ to be the maximum value such that there is a $h$-colouring of the elements of $M$ with no heterocromatic
     Hamiltonian circuit. By Lem\-ma~\ref{lemma:hamil_circuits}, $h\geq q+1$.  We are going to prove that $h\leq q+1$. 
    Let be $A$ a heterocromatic set of $h$ elements. As we do not have heterocromatic Hamiltonian circuits, each set of $r+1$
     elements in $A$ contains a $r$-circuit and these points are in a hyperplane. 
     
    Let us define a digraph for each set of $r+2$ points in $A$ in the following way: The vertices are the $r+2$ points, say 
    $x_1$, $x_2,\ldots x_{r+2}$. For each $x_i$, the remaining  $r+1$ points determined  a hyperplane $H$ that contains 
    at least $r$ of them. Let place the arrow from $x_i$ to the points, say $x_j$, that is not in $H$. If all the $r+1$ points 
    are in $H$, chose $x_j$ arbitrarily.  As $r+2$ is odd and all vertices have out degree 1, there is a cycle of size at least 3. 
    We can suppose, then that we have the arrows, $x_{r+1}x_{r+2}$ and $x_{r+2}x_{r}$. The first arrow implies that 
    $\{x_1,x_2,...x_{r-1}, x_{r}\}$ are in a hyperplane $H_1$, while the second one implies that 
    $\{x_1,x_2,...x_{r-1}, x_{r+1}\}$ are in a hyperplane $H_2$. As both hyperplanes share a, necessarily independent, 
    set of size $r-1$, we conclude that $H_1=H_2$, thus proving that in an heterocromatic set of size $r+2$ in $A$, 
    at least $r+1$ are always in a hyperplane.  

    Now, if for every set of $r+2$ points, they are in a hyperplane, all the elements in $A$ are in a hyperplane. 
    Because, if $\{x_1,x_2,...x_{r+1}, x_{r+2}\}$ are in a hyperplane $H_1$, and $w$ is another point in $A$, 
    then $\{x_1,x_2,...x_{r+1}, w\}$ are also in a hyperplane $H_2$. But $H_1$ and $H_2$ share at least $r+1$ points, 
    so $H_1=H_2$.

   So, there is a subset of $r+2$ points where $\{x_1,x_2,...x_{r}, x_{r+1}\}$ are in a hyperplane $H_1$ but 
   $x_{r+2}$ is not. In this case, let us take a point $w\in A$ and the set $B=\{x_1,x_2,...x_{r}, x_{r+2}, w\}$ of 
   $r+2$ points. As  $r+1$ of the points in $B$ are in a hyperplane $H_2$, these must include the points 
   $x_1,x_2,...x_{r-1}$ and $x_{r}$. Then, $H-1$ and $H_2$ share $r-1$ points at least and $H_1=H_2$. 
   As $x_{r+2}$ is not in $H_1=H_2$, then $w$ should be in $H_1=H_2$.
 
  In both cases we conclude that $A$ has at least $h-1$ points in a hyperplane, thus $h-1\leq q$ and the result follows.
  \end{proof}

 The general problem of finding the heterochromatic number of the hypergraph $C_p(M)$ for $M$ a rank-$r$ matroid and
  $C_p(M)$ the hypergraph with hyperedges the $p$-circuits of $M$ seems quite difficult, and even the case when $M$ is the
   graphic matroid of the complete graph took several years to be solved completely. Here we consider the case when $M$ is
    the projective geometry $PG(r-1,q)$. 
 
 The construction of a colouring with as many colours as possible but no heterochromatic edge in $C_p(M(K_n))$ given 
 in~\cite{ErSiSo1973} can be described in terms of the matroid $M(K_n)$ as follows: For a $3\leq p\leq n$, take a flat $F$ as 
 big as possible in $M(K_n)$ with no $p$-circuit and totally multicolored these elements. Now, consider a colouring of 
 $M(K_n)/F$ with a maximum number $s$ of colours and no totally multicoloured circuit. The claim is that $|F|+s$ is the
  heterochromatic number of $C_p(M(K_n))$ minus 1, i.e. it is the anti-Ramsey number $ar(n,C_p)$. 
  
  We propose a similar construction for $PG(r-1,q)$. Let $F$ be a flat of maximum size with the property of not containing a $p$-circuit as a restriction. Now, consider a colouring of $PG(r-1,q)/F$ with a maximum number $s$ of colours and no totally multicoloured circuit. This colouring does not have heterochromatic $p$-circuits, because the circuits of the matroid $M/F$ are the minimal non-empty sets in $\{C\setminus F: C \text{ is a circuit of}\ M\}$, so a $p$-circuit in $M$ contains a circuit of $M/F$. Thus, a $p$-circuit of $M$ contains a subset that is not totally multicoloured. 
   
   We conclude that the number $F|+s+1$ is a lower bound for $hc( C_p(PG(r-1,q)))$. In this particular case, $F$ is isomorphic to a projective geometry $PG(p-3,q)$, because all flats of a projective geometry are projective geometries and all projective geometries are Hamiltonian. Also, $PG(r-1,q)/F$ is isomorphic to a projective geometry $PG(r-p+2,q)$ with parallel elements but no loops. Then, our Theorem~\ref{Th:hete_circuit}  said that  $s=r-p+1$. This give a lower bound of 
   $\frac{q^{p-2}-1}{q-1}+(r-p+1)+1$ for the heterocromatic number of $p$-circuits.
   
   \begin{conjecture}
   Given intergers $3\leq p\leq r+1$ and a primer power $q$, the heterocromatic number of $C_p(PG(r-1,q))$ is 
   $\frac{q^{p-2}-1}{q-1}+(r-p+1)+1$
   \end{conjecture}

\section{Heterochromatic number and the Erd\"os-Stone theorem}
 In~\cite{GlNl2015}, Geelen and Nelson prove an analogue of the Erd\"os-Stone theorem for projective.
  geometries. For a set $M$ of points in $PG(m-1,q)$, for some fixed $m$, let $ex_q(M,n)$ denote the maximum size of a collection of points in $PG(n-1,q)$ not containing a copy of $M$, up to projective equivalence. They show that
  	\begin{equation}\label{eq:asymtotic_behaviour}
		\lim_{n\to \infty} \frac{ex_q(M;n)}{|PG(n-1,q)|}=1-q^{1-c},
	\end{equation}
  where $c=c(M)$ is the smallest no negative integer such that $\chi(M;q^c)>0$ and $\chi(M,\lambda)$ is the characteristic polynomial of the matroid. With this result in hand, we can proceed as in similar fashion as in~\cite{ErSiSo1973}. Let $H$ be a submatroid of $PG(m-1,q)$, for some fixed $m$, such that for all pair of elements $e,f\in H$, $H\setminus e\cong H\setminus f$. This would happen in many instances, for example when $H$ has transitive automorphism group. Let us consider the problem of finding the heterochromatic number of the hypergraph $S_H(PG(n-1,q))$, the hypergraph with elements the points of $PG(n-1,q)$ and the hyperedges, the elements of a copy of $H$ in $PG(n-1,q)$. We always  assume that $n$ is big enoung so that the hypergraph has non-empty hyperedges. 
  
  In order to bound the heterochromatic number of the hypergraph $S_H(PG(n-1,q))$, let us consider $H'=H\setminus e$ for some fixed $e\in H$. Take an extremal matroid $L^{-}$ of size $ex_q(H',n)$ in $PG(n-1,q)$.  the points in the complement respect to  $PG(n-1,q)$ is a 2-transversal in $S_H(PG(n-1,q))$, thus $ex_q(H',n)+1\leq hc(S_H(PG(n-1,q)))$. Second, let us consider $L^{+}=H\oplus_2 H$, the 2-sum of $H$ with itself with base point $e$. As $H$ is a submatroid of $PG(m-1,q)$, for some fixed $m$, the matroid $L^{+}$ is a submatroid of $PG(n-1,q)$, for a big enough $n$. If we colour $PG(n-1,q)$ with $ex_q(L^{+},n)+1$ colours, there is a totally multicoloured copy of $L^{+}$ and, independently of the colour that $e$ receives, there will be a multicoloured copy of $H$. Thus, $hc(S_H(PG(n-1,q)))\leq ex_q(L^{+},n)+1$.

  If the critical exponent of $H'=c'$, then the critical exponent of $H\oplus_2 H$ is also $c'$. This follows from the following result that is folklore. 
  
  \begin{lemma}
  For a loopless matroid $M$, the critical exponent of $M\oplus_2 M$ with basepoint $z$ equals $c(M\setminus z)$
  \end{lemma} 
  \begin{proof}
   Walton and Welsh in~\cite{WlWl1980}  give the following expression for the 2-sum of matroids $M_1$ and $M_2$ with basepoint $z$, 
  \begin{equation}
   \chi(M;\lambda)=\frac{\chi(M_1;\lambda)\chi(M_2;\lambda)}{\lambda-1}
                              +\chi(M_1/z;\lambda)\chi(M_2/z;\lambda).
\end{equation}
Form this we see that  $c(M\oplus_2 M)=\max\{c(M), c(M/z)\}$. Now, there are two basic results from~\cite{BO93} about critical exponent that will help us find the critical exponent of $M\oplus_2 M$. First, for any loopless matroid $M$ and a subset $T$ of $E(M)$, $C(M|T)\leq C(M)$. Second, for a a loopless matroid $M$ and a subset $S$ of $E(M)$ such that $c(M\setminus S)\leq k$,  there exists $T\subseteq S$ such that $C(M/T)\leq k$. For our case, 
if either $c(M\setminus e)=c(M)$ or $c(M\setminus e)=c(M/e)$, then $c(M\setminus e)=c(M\oplus_2 M)$. If $c(M\setminus e)<c(M)$ and $c(M\setminus e)<c(M/e)$, then we contradict the second property. 
\end{proof}

 Now, using Equation~\ref{eq:asymtotic_behaviour}, the inequality $ex_q(H',n)+1\leq hc(S_H(PG(n-1,q)))\leq ex_q(H\oplus_2 H,n)+1$ and that $H\setminus e$ and $H\oplus_2 H$ have the same critical exponent $c$, we get that 
 \begin{equation}\label{eq:asymtotic_behaviour_hc}
		\lim_{n\to \infty} \frac{hc(S_H(PG(n-1,q)))}{|PG(n-1,q)|}=1-q^{1-c},
	\end{equation}

\section{Conclusion}
 In this work we try to explore the natural generalization of the anti-Ramsey number for graphs to matroids, using the known generalization of heterochromatic number in hypergraphs. We manage to prove some of the  basic results from~\cite{ErSiSo1973} in the setting of projective geometries. We also obtain some results in the case of circuits and cocircuits for general matroids, thus generalizing the case of cycles and edge-cuts in graphs. We restrict to the clase of paving matroids and  consider the case of bases. Finally, we get partial results on the case of Hamiltonian circuits for matroids coming from Steiner systems.

 Another natural hypergraph coming from a matroid structure is the hypergraph which hyperedges are the flats of size $p$, let us denote this hypergraph by $F(M,p)$. Note that the $p$-flats of a paving matroid $M$ of rank $r$ consists of all the subsets of size $p$, when $p<r-1$, thus $hc( F(M,p))=p$. However, when $p=r-1$, that is, when we have the hypergraph of the hyperplanes of the matroid,  the situation is quite different and this problem included the heterocromatic number of flats of fixed size in $PG(n-1,q)$.  However, there has been development in the area of extremal matroid theory regarding this problem, for example, it is proved in~\cite{LiLuNlNo2017} that $ex_2(PG(t,2),n)=2^n(1-2^{-t})$ for all $1\leq t<n$.

\end{document}